\documentclass[a4paper,reqno]{amsart}

\usepackage{eucal}

\usepackage{amsthm}
\usepackage{amsmath}
\usepackage{amssymb}
\usepackage{pstricks}
\usepackage{pstricks-add}
\psset{nodesep=2pt,labelsep=1pt,arrowscale=1.2}
\usepackage{enumerate}
\usepackage[colorlinks=false,backref=true]{hyperref}

\theoremstyle{plain}
\newtheorem{theorem}{Theorem}[section]
\newtheorem{proposition}[theorem]{Proposition}
\newtheorem{lemma}[theorem]{Lemma}
\newtheorem{corollary}[theorem]{Corollary}
\theoremstyle{definition}
\newtheorem*{definition}{Definition}
\newtheorem*{problem}{Problem}
\newtheorem{example}[theorem]{Example}
\theoremstyle{remark}
\newtheorem*{remark}{Remark}

\newcommand{\injto}{\hookrightarrow}
\newcommand{\epito}{\twoheadrightarrow}
\newcommand{\restr}{\mathord{\upharpoonright}}
\newcommand{\End}{\mathrm{End}}
\newcommand{\Aut}{\mathrm{Aut}}
\newcommand{\Pol}{\mathrm{Pol}}
\newcommand{\Age}{\mathrm{Age}}
\newcommand{\pTp}{\operatorname{Tp}^{+}}
\newcommand{\pTpZ}{\operatorname{Tp}^{0}}
\newcommand{\Th}{\operatorname{Th}}
\newcommand{\bA}{\mathbf{A}}
\newcommand{\bB}{\mathbf{B}}
\newcommand{\bBB}{\mathbf{B_1}}
\newcommand{\bBBB}{\mathbf{B_2}}
\newcommand{\bC}{\mathbf{C}}
\newcommand{\bD}{\mathbf{D}}
\newcommand{\bE}{\mathbf{E}}
\newcommand{\bF}{\mathbf{F}}
\newcommand{\bH}{\mathbf{H}}
\newcommand{\bG}{\mathbf{G}}
\newcommand{\cA}{\mathcal{A}}
\newcommand{\cB}{\mathcal{B}}
\newcommand{\cC}{\mathcal{C}}
\newcommand{\cR}{\mathcal{R}}
\newcommand{\cP}{\mathcal{P}}
\newcommand{\cD}{\mathcal{D}}
\newcommand{\cdom}{\operatorname{cdom}}
\newcommand{\img}{\operatorname{img}}
\newcommand{\dom}{\operatorname{dom}}
\newcommand{\CSP}{\operatorname{CSP}}
\newcommand{\Sym}{\operatorname{Sym}}
\newcommand{\Fraisse}{Fra\"{\i}ss\'e}

\title[Towards a Ryll-Nardzewski-type Theorem for w.\ o.\ structures]{Towards a Ryll-Nardzewski-type Theorem for weakly oligomorphic structures}
\author{Christian Pech}
\address{Christian Pech\\Institut f\"ur Algebra\\Technische Universit\"at Dresden\\01062 Dresden\\Germany}
\email{Christian.Pech@tu-dresden.de}
\author{Maja Pech}
\address{Maja Pech\\Institut f\"ur Algebra\\Technische Universit\"at Dresden\\01062 Dresden\\Germany\\ and Department of Mathematics and Informatics\\University of Novi Sad\\Trg Dositeja Obradovi\'ca 4\\ 21000 Novi Sad\\Serbia}
\thanks{supported by the Ministry of Education and Science of the Republic of Serbia through Grant No.174018,  by the grant (Contract 114-451-1901/2011) of the the Secretariat of Science and Technological Development of the Autonomous Province of Vojvodina, and DAAD reinvitation scholarship.}
\email{Maja.Pech@tu-dresden.de, maja@dmi.uns.ac.rs}

\begin{document}
\maketitle

\begin{abstract}
	A structure is called weakly oligomorphic if it realizes only finitely many $n$-ary positive existential types for every $n$. The goal of this paper is to show that the notions of homomorphism-homogeneity, and weak oligomorphy are not only completely analogous to the classical notions of ultrahomogeneity and oligomorphy, but are actually closely related.  

	A first result is a \Fraisse-type theorem for homomorphism-homogeneous relational structures. 

	Further we show that every weakly oligomorphic homomorphism-homogeneous structure contains (up to isomorphism) a unique homogeneous, homomorphism-homogeneous core, to which it is homomorphism-equivalent. As a consequence, we obtain that every countable weakly oligomorphic structure is homomorphism-equivalent with a finite or $\omega$-categorical structure.  

	Another result is the characterization of positive existential theories of weakly oligomorphic structures as the positive existential parts of $\omega$-categorical theories. Finally, we show, that the countable models of countable weakly oligomorphic structures are mutually homomorphism-equivalent (we call first order theories with this property weakly $\omega$-categorical). These results are in analogy with part of the Engeler-Ryll-Nardzewski-Svenonius-theorem. 
\end{abstract}

\section{Introduction}
\label{sec:introduction}
The notion of oligomorphic permutation groups goes back to Peter Cameron, who introduced it in the 1970s. They create a bridge between such diverse fields of mathematics like permutation group theory, enumerative combinatorics, and model theory (see \cite{Cam90}). A permutation group $G\le\Sym(A)$ is called \emph{oligomorphic} if it has only finitely many $n$-orbits for every $n$. 

There is a close relationship between oligomorphic permutation groups and homogeneous structures. A structure is called homogeneous if every isomorphism between finitely generated substructures can be extended to an automorphism. For instance, every  countable homogeneous relational structure over a finite signature has an oligomorphic automorphism group, and is hence $\omega$-categorical. On the other hand, every oligomorphic permutation group on a countable set is the automorphism group of a suitable structure, and every countable structure that has an oligomorphic automorphism group can be expanded to a homogeneous structure (though, not necessary over a finite signature). 

In their seminal paper \cite{CamNes06}, Peter Cameron and Jaroslav Ne\v{s}et\v{r}il introduced several variations to the concept of homogeneity, one of them being homomorphism-homogeneity---saying that every homomorphism between finitely generated substructures of a given structure extends to an endomorphism of that structure. The relevance of this notion in the theory of transformation monoids on countable sets was realized quickly \cite{Dol12,Dol12b,MPPhD,AUpaper,Pon05}. Also a classification theory for homomorphism-homogeneous structures emerged quickly \cite{CamLoc10,DolMas11,IliMasRaj08,JunMas12,Mas07,Mas12,MasNenSko11} and classes of high complexity of finite homomorphism-homogeneous structures were discovered \cite{Mas12b,RusSch10}.

Weak oligomorphy is a phenomenon that arrises naturally in the context of homomorphism-homogeneity. A countable relational structure is  weakly oligomorphic if its endomorphism monoid is oligomorphic, i.e., it has of every arity only finitely many invariant relations.  It is not hard to see that every homomorphism-homogeneous relational structure over a finite signature is weakly oligomorphic.  Moreover,  every weakly oligomorphic relational structure has a positive existential expansion that is homomorphism-homogeneous (cf. \cite{MasPec11,MPPhD,AUpaper}). Clearly, every oligomorphic structure is weakly oligomorphic, but the reverse does not hold --- e.g., there exist countably infinite homomorphism-homogeneous graphs that have a trivial automorphism group (cf. \cite[Cor. 2.2]{CamNes06}). However, the graph signature is finite and hence such graphs are weakly oligomorphic. 

Cameron and Ne\v{s}et\v{r}il in \cite{CamNes06} posed the problem to understand the phenomenon of weak oligomorphy. Of particular interest is, to characterize the (countable) structures with an oligomorphic endomorphism monoid. First steps in solving this problem were done in \cite{MasPec11} and in \cite{AUpaper}.  The main goal of this paper is to to understand further the nature of weakly oligomorphic structures, and thus to come closer to a  satisfying answer to the problem by Cameron and Ne\v{s}et\v{r}il. To this end we will create cross-links between the theory of weakly oligomorphic structures and the theory of oligomorphic structures, as well as between the theory of homomorphism-homogeneous structures and the theory of homogeneous structures.  

In Section~\ref{sec:weakly-olig-struct}, we define the notion of weak oligomorphy and we show that from an algebraic  point of view, this is the weakest reasonable relaxation of the notion of oligomorphy. 

In Section~\ref{sec:homom-homog-amalg} we give a characterization of the ages of homomorphism-homogeneous structures in the vein of \Fraisse's Theorem by showing that a class of finite structures is the age of a homomorphism-homogeneous structure if and only if it is a homo-amalgamation class. 

In Section~\ref{sec:ages-weakly-olig} we show that the existence of a homomorphism between countable weakly oligomorphic structures depends only on a relation between their ages.  

In Section~\ref{sec:cores-homom-homog} we show that every countable weakly oligomorphic homomorphism-homogeneous structure has a unique (up to isomorphism) homomorphism-equivalent substructure that is oligomorphic, homogeneous, and a core.

In Section~\ref{sec:omega-categ-substr} we show that every countable weakly oligomorphic structure is homomorphism-equivalent to a finite or $\omega$-categorical structure.   

In Section~\ref{sec:posit-exist-theor}, the positive existential theories of countable weakly oligomorphic structures are characterized as the positive existential parts of $\omega$-categorical theories. Moreover, it is shown that the countable models of the first order theories of countable weakly oligomorphic structures are mutually homomorphism-equivalent. This, together with \cite[Thm.3.5]{MasPec11} gives an almost complete analogue of the Engeler-Ryll-Nardzewski-Svenonius Theorem for weakly oligomorphic structures.

\section{Preliminaries}
\label{sec:preliminaries}

The main object of study in this paper are relational structures. As a basis for our notions and notations we use Hodges' \cite{Hod97}. A relational signature is a model-theoretic signature without constant- and function symbols. A model over a relational signature  will be called a \emph{relational structure}. Note, that throughout this paper we make no other assumptions about the signatures. In particular, if not stated otherwise, we allow signatures of any cardinality. Relational structures will be denoted by bold capital letters $\bA$, $\bB$, etc., while their carriers will be denoted by usual capital letters $A$, $B$, etc., respectively.

As usual, a homomorphism between two relational structures is a function between the carriers that preserves all relations. We will use the notation $\bA\to\bB$ as a way to say that there exists a homomorphism from $\bA$ to $\bB$. If $\bA\to\bB$ and $\bB\to\bA$, then we call $\bA$ and $\bB$ \emph{homomorphism-equivalent}. It is easy to see, that homomorphism-equivalent relational structures define equivalent constraint satisfaction problems.

If $f:\bA\to\bB$, then we call $\bA$ the \emph{domain} of $f$ (denoted by $\dom(f)$), and $\bB$, the \emph{codomain} (denoted by $\cdom(f)$). Moreover, the structure induced by $f(A)$ is called the \emph{image} of $f$ (denoted by $\img(f)$). 

\emph{Epimorphisms} are surjective homomorphisms, \emph{monomorphisms} are injective homomorphisms, and isomorphisms are bijective homomorphisms whose inverse is a homomorphism, too. \emph{Embeddings} are monomorphisms that not only preserve relations but also reflect them. That is, a monomorphism is an embedding if and only if it is an isomorphism to its image. 

As a final note, in this paper under a countable set we understand a finite or countably infinite set.

\section{Weakly oligomorphic structures}\label{sec:weakly-olig-struct}
In \cite{Cam90}, Peter Cameron introduced the notion of oligomorphic permutation groups (though the notion was created already in the 1970s).  A structure $\bA$ is called \emph{oligomorphic} if its automorphism group is oligomorphic.

Before coming to the definition of weakly oligomorphic structures, we have to recall some model theoretic notions: Let $\Sigma$ be a relational signature, and let $L(\Sigma)$ be the language of first order logics with respect to $\Sigma$. Let $\bA$ be a $\Sigma$-structure. For a formula $\varphi(\bar{x})$ (where $\bar{x}=(x_1,\dots,x_n)$), we define $\varphi^\bA\subseteq A^n$ as the set of all $n$-tuples $\bar{a}$ over $A$ such that $\bA\models\varphi(\bar{a})$. More generally, for a set $\Phi$ of formulae from $L$ with free variables from $\{x_1,\dots,x_n\}$, we define $\Phi^\bA$ as the intersection of all relations $\varphi^\bA$ where $\varphi$ ranges through $\Phi$. We call $\Phi$ a \emph{type}, and $\Phi^\bA$ the relation defined by $\Phi$ in $\bA$.

If $\Phi^\bA\neq\emptyset$, then we say that $\bA$ realizes $\Phi$. We call $\Phi$ \emph{positive existential}, or \emph{positive primitive}, if it consists just of positive existential, or positive primitive formulae, respectively. 

For a relation $\varrho\subseteq A^n$ by $\pTp_\bA(\varrho)$ we denote the set of all positive existential formulae $\varphi(\bar{x})$ such that $\varrho\subseteq \varphi^\bA$. This is the positive existential type defined by $\varrho$ with respect to $\bA$. With $\pTpZ_\bA(\varrho)$ we will denote the quantifier free part of $\pTp_\bA(\varrho)$.

Let us now come to the definition of the structures under consideration in this paper. 
\begin{definition}[\cite{MPPhD,AUpaper}] 
	A relational structure $\bA$ is called \emph{weakly oligomorphic} if for every arity there are just finitely many relations that can be defined by positive existential types.
\end{definition}
One can argue that it would be more appropriate to define a structure $\bA$ to be weakly oligomorphic if its endomorphism monoid is \emph{oligomorphic} (i.e.\ there are just finitely many invariant relations of $\End (\bA)$ of any arity). However, there is no need to worry, since, at least for countable structures, these two definitions are equivalent:
\begin{proposition}[{\cite[Prop. 2.2.5.1]{MPPhD}}, {\cite[Thm. 6.3.4]{AUpaper}}]\label{prop:weakly:olig}
	A countable structure $\bA$ is weakly oligomorphic if and only if $\End(\bA)$ is oligomorphic.  
\end{proposition}
Clearly, if a structure is oligomorphic, then it is also weakly oligomorphic.

In principle, it is possible to introduce an even weaker notion of weak oligomorphy. Recall that \emph{an $n$-ary polymorphism} of a structure $\bA$ is a homomorphism from $\bA^n$ to $\bA$. The collection of all polymorphisms of $\bA$ forms a clone (cf. \cite{PoeKal79,Sze86}) and is denoted by $\Pol(\bA)$. We could call $\Pol(\bA)$ oligomorphic if it has of every arity just finitely many invariant relations. However, the following Proposition will show that (at least in the countable case) this does not lead to a strictly weaker notion.
\begin{proposition}[{cf. \cite[Thm.2.8]{MasWOC}}]
	Let $\bA$ be a countable relational structure. Then the following are equivalent:
	\begin{enumerate}[(a)]
		\item\label{item:7} $\Pol(\bA)$ has of every arity just finitely many invariant relations,
		\item\label{item:8} for every arity there are just finitely many relations that can be defined by positive primitive types,
		\item\label{item:9} $\bA$ is weakly oligomorphic.
	\end{enumerate}
\end{proposition}
\begin{proof}
	((\ref{item:7})$\Rightarrow$(\ref{item:8})) Every relation definable by a set of positive primitive formulae is invariant under $\Pol(\bA)$. Hence, the set of definable relations of a given arity is a subset of the set of invariant relations of this arity.

	((\ref{item:8})$\Rightarrow$(\ref{item:9}) Every positive existential formula is equivalent to a finite disjunction of positive primitive formulae. Thus, the set of relations of arity $n$ definable by positive existential formulae is obtained from the set of relations definable by positive primitive formulae by closing the latter one with respect to finite unions. If there are just finitely many relations definable by sets of positive primitive formulae, then the closure with respect to finite unions will not produce an infinite set of relations.

	((\ref{item:9})$\Rightarrow$(\ref{item:7})) From Proposition~\ref{prop:weakly:olig}, it follows that $\End(\bA)$ is  oligomorphic. However, $\End(\bA)\subset\Pol(\bA)$. Hence the set of invariant relations of $\Pol(\bA)$ is contained in the set of invariant relations of $\End(\bA)$. This finishes the proof.
\end{proof}

\section{Homomorphism-homogeneity and amalgamation}\label{sec:homom-homog-amalg}
Recall that \emph{the age of a relational structure} is the class of finite structures
embeddable into it. 
\begin{definition}
	Let $\cC$ be a class of finite relational structures over the same signature. We say that $\cC$ has the
	\begin{description}
		\item[Joint Embedding Property (JEP)] if for all $\bA,\bB\in \cC$, there exists a $\bC\in\cC$ such that both $\bA$ and $\bB$ are embeddable into $\bC$: 
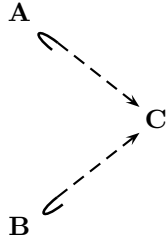
		\[
		\psmatrix[rowsep=1cm]
		\bA &\\
		&\bC \\
		\bB &
    	\endpsmatrix
		\everypsbox{\scriptstyle}
		\ncline[linestyle=dashed]{H->}{1,1}{2,2}
		\ncline[linestyle=dashed]{H->}{3,1}{2,2}
		\]
		\item[Hereditary Property (HP)] if for all $\bA\in \cC$, and for all $\bB< \bA$, we have that $\bB\in\cC$,
		\item[Amalgamation property (AP)] if for all $\bA, \bBB,\bBBB \in \cC$, and for all embeddings $f_1:\bA\injto \bBB$, $f_2:\bA\injto \bBBB$ there exist $\bC\in \cC$, and embeddings $g_1:\bBB\injto \bC$ and $g_2:\bBBB\injto\cC$ such that the following diagram commutes: 
		\[
		\begin{psmatrix}
		\bA & \bBB\\
		\bBBB &\bC, 
		\everypsbox{\scriptstyle}
		\ncline{H->}{1,1}{1,2}^{f_1}
		\ncline{H->}{1,1}{2,1}<{f_2}
		\ncline[linestyle=dashed]{H->}{2,1}{2,2}^{g_2}
		\ncline[linestyle=dashed]{H->}{1,2}{2,2}<{g_1}
		\end{psmatrix}
		\]
		\item[Homo-Amalgamation Property (HAP)] if for all $\bA, \bBB,\bBBB \in \cC$, $f_1:\bA\to \bBB$, and for every embedding  $f_2:\bA\injto \bBBB$ there exist $\bC\in \cC$, an embedding $g_1:\bBB\injto \bC$ , and a homomorphism $g_2:\bBBB\to\cC$ such that the following diagram commutes: 
		\[
		\begin{psmatrix}
		\bA & \bBB\\
		\bBBB &\bC. 
		\everypsbox{\scriptstyle}
		\ncline{->}{1,1}{1,2}^{f_1}
		\ncline{H->}{1,1}{2,1}<{f_2}
		\ncline[linestyle=dashed]{->}{2,1}{2,2}^{g_2}
		\ncline[linestyle=dashed]{H->}{1,2}{2,2}<{g_1}
		\end{psmatrix}
		\]
	\end{description}
\end{definition}

A basic theorem by Roland \Fraisse{} states that a class of finite structures of the same type is the age of a countable structure if and only if
\begin{enumerate}
	\item it is isomorphism-closed,
	\item it has only countably many isomorphism types,
	\item it has the  HP and the JEP.
\end{enumerate}

Another central result of \Fraisse{} is the characterization of the ages of homogeneous structures. We quote the formulation due to Cameron (cf. \cite[(2.12-13)]{Cam90}):
\begin{theorem}[\Fraisse{} (\cite{Fra53})]
	A class $\cC$ of finite relational structures is the age of some countable homogeneous relational structure if and only if
	\begin{enumerate}[(i)]
	\item it is closed under isomorphism,
	\item it has only countably many isomorphism types,
	\item it has the  HP and the AP.
	\end{enumerate}
	Moreover, any two countable homogeneous relational structures with the same age are isomorphic.\qed
\end{theorem}
A class of finite structures over the same signature, that is isomorphism-closed and that has the  HP, and the AP, is called \emph{\Fraisse-class}.

In \cite{CamNes06}, Peter Cameron and Jaroslav Ne\v set\v ril introduced the notion of homomorphism-homogeneous structure. A \emph{local homomorphism} of a structure $\bA$ is a homomorphism from a finite substructure of $\bA$ to $\bA$.
\begin{definition}[Cameron, Ne\v set\v ril]
	A structure $\bA$ is called \emph{homomorphism-homogeneous} if every local homomorphism of $\bA$ can be extended to an endomorphism of $\bA$.
\end{definition}

As a straight forward adaptation of the notion of weakly homogeneous structure (cf. \cite[Sec.6.1]{Hod97}), it will be useful to introduce the notion of a weakly homomorphism-homogeneous structure:
\begin{definition}
	A structure $\bA$ is called \emph{weakly homomorphism-homogeneous} if whenever $\bB<\bC$ are finite substructures of $\bA$, then every homomorphism $f:\bB\to\bA$ extends to $\bC$.
\end{definition}
Clearly, a countable structure is weakly homomorphism-homogeneous if and only if it is homomorphism-homogeneous.

The concept of weak homomorphism-homogeneity is closely related to other model theoretic notions, via the notion of weak oligomorphy. 

\begin{definition}
	A structure $\bA$ is called \emph{endolocal} if for any relation $\varrho\subseteq A^n$ we have that $\varrho$ is definable by a positive existential type if and only if for all  $\bar{a}\in \varrho$, $\bar{b}\in A^n$ holds that if $\pTpZ_\bA(\bar{a})\subseteq \pTpZ_\bA(\bar{b})$, then $\bar{b}\in\varrho$. 
\end{definition}
Another, perhaps more transparent definition of endolocality is the following:
\begin{lemma}
	A structure $\bA$ is endolocal if and only if the set of relations definable by positive existential types in $\bA$ coincides with the set of relations that are invariant under local homomorphisms of $\bA$.\qed
\end{lemma}

The following characterization of weak homomorphism-homogeneity for weakly oligomorphic structures will be of good help in later proofs. It first appeared in \cite{MPPhD}.
\begin{proposition}[{\cite[Thm.6.1]{AUpaper}}]\label{MainThm}
	Let $\bA$ be a weakly oligomorphic structure. Then the following are equivalent:
	\begin{enumerate}[(a)]
	\item $\bA$ is weakly homomorphism-homogeneous,
	\item $\bA$ is endolocal,
	\item every positive existential formula in the language of $\bA$ is equivalent in $\bA$ to a quantifier free positive formula.\qed
	\end{enumerate}
\end{proposition}

We continue with an analogue of \Fraisse's Theorem for homomorphism-homogeneous structures:
\begin{theorem}
	\begin{enumerate}[(a)]
	\item\label{item:1} The age of any homomorphism-homogeneous structure has the HAP.
	\item\label{item:2} If a class $\cC$ of finite relational structures is isomorphism-closed, has only a countable number of isomorphism types, and has the HP, the JEP, and the HAP, then there is a countable homomorphism-homogeneous structure $\bH$ whose age is equal to $\cC$.
	\end{enumerate}
\end{theorem}
\begin{proof}
\begin{description}
	\item[About (\ref{item:1})] Let $\bH$ be a homomorphism-homogeneous relational structure and let $\cC$ be its age. Let  $\bA, \bBB,\bBBB\in\cC$, and let $f_1:\bA\to \bBB$ be a homomorphism and $f_2:\bA\to \bBBB$ an embedding. Without loss of generality, we can assume that $\bA\leq \bBBB$, $f_2$ is the identical embedding, and $\bBB, \bBBB\leq \bH$. 

	Since $f_1$ is also a homomorphism from $\bA$ to $\bH$, and $\bH$ is homomorphism-homogeneous, it follows that $f_1$ can be extended to a $g\in\End (\bH)$. 

	Further, we define  
	\begin{itemize}
		\item $g_2:=g\restr_{\bBBB}$,
  		\item $\bC:=\img(g_2)\cup \bBB$, and
  		\item $g_1$ to be the identical embedding from $\bBB$ to $\bC$.
	\end{itemize}
	Then we obtain the following commuting diagram
	\[
	\begin{psmatrix}
	\bA & \bBB\\
	\bBBB &\bC 
	\everypsbox{\scriptstyle}
	\ncline{->}{1,1}{1,2}^{f_1}
	\ncline{H->}{1,1}{2,1}<{\leq}>{f_2}
	\ncline{->}{2,1}{2,2}^{g_2}
	\ncline{H->}{1,2}{2,2}<{\leq}>{g_1}
	\end{psmatrix}
	\]
	Indeed, for $a\in A$ we have
	\[
	g_1(f_1(a))=f_1(a)=g(a)=g_2(a)=g_2(f_2(a)),
	\] 
	so this diagram commutes, and hence the age of $\bH$ has the HAP.
	\item[About (\ref{item:2})] Our goal is to  construct a countable homomorphism-homogeneous structure $\bH$ whose age is equal to $\cC$.

	Since $\cC$ has a countable number of isomorphism classes, we can choose a representative from each class thus obtaining a countable set of structures. Denote this set by $\cR$ and well-order $\cR$ like $\omega$. 

	We aim to construct a chain $\bH_i$, $i\in \mathbb{N}$, of structures from $\cC$ such that the following holds: 
	\begin{enumerate}[(I)]
		\item\label{prvi} If $\bA,\bB\in\cC$, where $\bA<\bB$, then for each homomorphism $f:\bA\to \bH_i$, for some $i\in \mathbb{N}$, there are $j>i$ and a homomorphism $g:\bB\to\bH_j$ which extends $f$. 
		\item\label{drugi} For every $\bA\in\cC$ there exists an $i\in \mathbb{N}$ such that $\bA$ is embeddable in $\bH_i$.
	\end{enumerate}
	We claim that the union $\bH=\cup_{i\in\mathbb{N}}\bH_i$ of such a chain of structures is going to be a countable homomorphism-homogeneous structure with age $\cC$.

	First of all, note that if for each $i\in\mathbb{N}$, the age of $\bH_i$ is included in $\cC$, then the age of $\bH=\cup_{i\in\mathbb{N}} \bH_i$ is also included in $\cC$. 

	On the other hand, take some $\bA\in\cC$. Then from (\ref{drugi}) it follows that $\bA$ is embeddable into some $\bH_i$, and, therefore, also into $\bH$, showing that it is in the age of $\bH$.

	It is left to show that $\bH$ is weakly homomorphism-homogeneous. Let $\bA<\bB$ be two finite substructures of $\bH$, and let $f:\bA\to\bH$ be a local homomorphism.  Then, since $f(\bA)$ is finite and $\bH=\bigcup_{i\in\mathbb{N}} \bH_i$, there exists some $i\in\mathbb{N}$ such that $f(\bA)<\bH_i$. Thus, by (\ref{prvi}) there is a $j>i$ and a homomorphism $g:\bB\to\bH_j$ that extends $f$, i.e. the following diagram commutes:
	\[
	\begin{psmatrix}
	\bB & \bH_j & \bH\\
	\bA &\bH_i &
	\everypsbox{\scriptstyle}
	\ncline{->}{1,1}{1,2}^{g}
	\ncline{H->}{1,2}{1,3}^{\le}
	\ncline{H->}{2,1}{1,1}<{\le}
	\ncline{H->}{2,2}{1,2}<{\le}
	\ncline{->}{2,1}{2,2}^{f}
	\ncline{H->}{2,2}{1,3}>{\le}
	\end{psmatrix}
	\]
	This implies that $\bH$ is weakly homomorphism-homogeneous. 

	To conclude, the structure $\bH$, which is the union of the chain of structures fulfilling conditions (\ref{prvi}) and (\ref{drugi}), is weakly homomorphism-homogeneous (and thus, since it is countable, also homomorphism-homogeneous), and it has $\cC$ as its age.

	It remains to construct the chain. In addition to the above defined set $\cR$, we define set $\cP$ of pairs of structures $(\bA,\bB)$ such that $\bA,\bB\in\cC$, and $\bA\leq \bB$. We choose $\cP$ so that it contains representatives from each isomorphism class of such pairs. Hence, $\cP$ is a countable set. 
	Now choose a bijection $\pi :2\mathbb{N}\times \mathbb{N}\rightarrow \mathbb{N}$ such that for every $i\in 2\mathbb{N}, j\in \mathbb{N}$ holds  $\pi(i,j)\geq i$. Then the construction goes by induction as follows: 

	Take the first structure from $\cR$ and denote it by $\bH_0$.

	\begin{enumerate}[\bfseries Step 1:]
	\item\label{korak1} Suppose that we have constructed $\bH_k$, for $k$ even. We proceed as follows: For given $\bH_k$ we list all the triples $(\bA_{kl},\bB_{kl},f_{kl})$, where $(\bA_{kl},\bB_{kl})\in \cP$ and $f_{kl}:\bA_{kl}\rightarrow \bH_k$ is a homomorphism (there are just countably many such triples).  
  
	In the next step, we find $i$ and $j$ such that $\pi(i,j)=k$. Since $i\leq \pi(i,j)=k$, it follows that the triple $(\bA_{ij},\bB_{ij},f_{ij})$ was already determined as a member of the list for $\bH_i$ in one of the previous iterations of Step~\ref{korak1}. We apply now to this triple HAP: 
	\[
	\begin{psmatrix}
	\bA_{ij} & \bH_i & \bH_k\\
	\bB_{ij} & & \bH_{k+1}
	\everypsbox{\scriptstyle}
	\ncline{->}{1,1}{1,2}^{f_{ij}}
	\ncline{H->}{1,1}{2,1}<{\leq}
	\ncline{H->}{1,2}{1,3}^{\leq}
	\ncline{H->}{1,3}{2,3}<{\leq}
	\ncline{->}{2,1}{2,3}^{g_{ij}}
	\end{psmatrix}
	\]
	with $g_{ij}\restr_{\bA_{ij}}=f_{ij}$, and obtain the structure $\bH_{k+1}$. (Note that $\bH_{k+1}$ can be always chosen in such a way that $\bH_k\leq \bH_{k+1}$ by taking for $\bH_{k+1}$ the appropriate structure from the isomorphism class, i.e.\ by changing the representative of the class.) 
	\item\label{korak2} Suppose that we have constructed $\bH_k$, for $k$ odd. Take the next structure $\bA\in\cR$. By JEP, there exists $\bH_{k+1}$ such that both $\bH_k$ and $\bA$ are embeddable into it: 
	\[
	\begin{psmatrix}[rowsep=1cm]
	\bA & \\
	& \bH_{k+1} \\
	\bH_k  &
	\everypsbox{\scriptstyle}
	\ncline{H->}{1,1}{2,2}
	\ncline{H->}{3,1}{2,2}
	\end{psmatrix}
	\]
	\end{enumerate}
	Together, Step~\ref{korak1} and Step~\ref{korak2} form the induction step. 

	This completes the description of the construction. It is left to prove that the constructed chain fulfills conditions \eqref{prvi} and \eqref{drugi}: 
	\begin{description}
	\item[About \eqref{prvi}]
	We identify a given triple $(\bA,\bB,f)$ in the list for $\bH_i$ in the construction. Let it be $(\bA_{il},\bB_{il},f_{il})$. Compute $j:=\pi(i,l)$. The requested $g$ and $\bH_j$ are those that appear on the diagram for the construction of $\bH_{j+1}$: 
	\[
	\begin{psmatrix}
	\bA & & &\\
	&\bA_{il} & \bH_i & \bH_{j}\\
	&\bB_{il} & & \bH_{j+1}\\
	\bB & & &
	\everypsbox{\scriptstyle}
	\ncline{->}{2,2}{2,3}_{f_{il}}
	\ncline{H->}{2,2}{3,2}<{\leq}
	\ncline{H->}{2,3}{2,4}_{\leq}
	\ncline{H->}{2,4}{3,4}<{\leq}
	\ncline{->}{3,2}{3,4}^{g_{il}}
	\ncline{H->}{1,1}{4,1}<{\leq}
	\ncline{->}{1,1}{2,2}<{\cong}
	\ncline{->}{1,1}{2,3}^{f}
	\ncline{->}{4,1}{3,2}^{\cong}
	\ncline{->}{4,1}{3,4}^{g}
	\end{psmatrix}
	\]
	\item[About \eqref{drugi}]
	Follows immediately from Step~\ref{korak2} in the construction.
	\end{description}
	\end{description}
\end{proof}
A class of finite relational structures over the same signature, that is closed under isomorphism and that has the HP, the JEP, and the HAP will be called \emph{homo-amalgamation class}.

Note that, in contrast to \Fraisse's construction, the construction of a homomorphism-homogeneous structure from an age does not guaranty the uniqueness of the result up to isomorphism. Indeed, all countably infinite linear orders are homomorphism-homogeneous and have as the age the class of all finite linear orders. Let us therefore have a look, how the different homomorphism-homogeneous structures with the same age are interrelated.
\begin{definition}
	Let $\bH$ and $\bH'$ be two relational structures. We write $\bH\preceq_h \bH'$ if
	\begin{itemize}
		\item $\Age (\bH)\supseteq \Age (\bH')$, and
		\item for all finite $\bA\leq \bB\leq \bH$ we have that every homomorphism from $\bA$ to $\bH'$ extends to a homomorphism from $\bB$ to $\bH'$.
	\end{itemize}
\end{definition}
An easy observation is:
\begin{lemma}
	If $\bH$ is countable and  $\bH\preceq_h \bH'$, then $\bH\to\bH'$.\qed
\end{lemma}
\begin{proposition}
Let $\bH$ and $\bH'$ be two relational structures.
\begin{enumerate}[(a)]
\item\label{item:3} If $\bH\preceq_h \bH'$, and $\bH$ is weakly homomorphism-homogeneous, then $\bH'$ is weakly homomorphism-homogeneous, too.
\item\label{item:4} If $\bH'$ is weakly homomorphism-homogeneous, and $\Age (\bH)=\Age (\bH')$, then $\bH\preceq_h \bH'$. 
\end{enumerate}
\end{proposition}
\begin{proof}
\begin{description}
\item[About (\ref{item:3})] Let $\bA'$, $\bB'$ be finite substructures
  of $\bH'$, such that $\bA'\leq \bB'$, and let $f':\bA'\rightarrow
  \bH'$ be a local homomorphism. 
Since $\Age (\bH)\supseteq \Age (\bH')$, and $\bA',\bB'\in\Age
(\bH')$, it follows that there is a $\bB\leq \bH$, such that $\bB\cong
\bB'$. 
Further, since $\bA'\leq \bB'$, it follows that there is an $\bA\leq \bB$, such that $\bA\cong \bA'$, and that the following diagram commutes:
\[
\begin{psmatrix}
  \bA & \bA'\\
  \bB & \bB'
  \everypsbox{\scriptstyle}
  \ncline{H->}{1,1}{2,1}<{\leq}
  \ncline{H->}{1,2}{2,2}<{\leq}
  \ncline{->}{1,1}{1,2}^{\varphi}_{\cong}
  \ncline{->}{2,1}{2,2}^{\psi}_{\cong}
\end{psmatrix}
\]
where $\psi$ is any isomorphism, $\varphi$ is obtained by restricting
$\psi$. Since $\bH\preceq_h\bH'$, the homomorphism
$f'\circ\varphi: \bA\to\bH'$ has an extension $g$ to $\bB$, as
depicted in the following commuting diagram:
\begin{equation}\label{eq:2}
\raisebox{\dimexpr-\height+\ht\strutbox}{%
\ensuremath{\begin{psmatrix}
  \bA & \bA' & \\
  \bB & \bB' & \\
      &      & \bH'
  \everypsbox{\scriptstyle}
  \ncline{H->}{1,1}{2,1}<{\leq}
  \ncline{H->}{1,2}{2,2}<{\leq}
  \ncline{->}{1,1}{1,2}^{\varphi}_\cong
  \ncline{->}{2,1}{2,2}^{\psi}_\cong
  \ncline{->}{1,2}{3,3}^{f'}
  \ncline{->}{2,1}{3,3}_{g}
\end{psmatrix}}}
\end{equation}
Now we define $g':=g\circ\psi^{-1}$. Using diagram~\eqref{eq:2}, it can easily be checked that this is an
extension of $f'$ to $\bB'$. 
This  shows that $\bH'$ is weakly homomorphism-homogeneous.
\item[About (\ref{item:4})] It suffices to show that for a given finite
  structure $\bA$ with $\bA\leq\bB\leq\bH$, and a homomorphism
  $f:\bA\rightarrow \bH'$, we can always extend $f$ to a homomorphism
  from $\bB$ to $\bH'$.

Note that $\bA,\bB\in\Age(\bH)=\Age(\bH')$. Then there exists a
$\bB'\leq \bH'$, such that $\bB\cong \bB'$. Let $\bA'\leq
\bB'$ such that $\bA'\cong \bA$, and let
$\varphi:\bA\to\bA'$, $\psi:\bB\to\bB'$ be isomorphisms that make the following diagram
commutative: 
\[
\begin{psmatrix}
  \bA & \bA'\\
  \bB & \bB'
  \everypsbox{\scriptstyle}
  \ncline{->}{1,1}{1,2}^{\varphi}_{\cong}
  \ncline{->}{2,1}{2,2}^{\psi}_{\cong}
  \ncline{H->}{1,1}{2,1}<{\leq}
  \ncline{H->}{1,2}{2,2}<{\leq}
\end{psmatrix}
\]
Let $f':=f\circ\varphi^{-1}$. Then $f'$ is a homomorphism from $\bA'$ to $\bH'$. Since $\bH'$ is weakly homomorphism-homogeneous, there exists a homomorphism $g':\bB'\to \bH'$ that extends $f'$:
\[
\begin{psmatrix}
         & &\bH'\\
  \bA & \bA'&\\
  \bB & \bB'&
  \everypsbox{\scriptstyle}
  \ncline{->}{2,1}{2,2}^{\varphi}_\cong
  \ncline{->}{3,1}{3,2}^{\psi}_\cong
  \ncline{H->}{2,1}{3,1}<{\leq}
  \ncline{H->}{2,2}{3,2}<{\leq}
  \ncline{->}{2,1}{1,3}^{f}
  \ncline{->}{2,2}{1,3}_{f'}
  \ncline{->}{3,2}{1,3}_{g'}
\end{psmatrix}
\]
We define $g:=g'\circ\psi$. It is easily checked that $g$ extends $f$ to a
homomorphism from $\bB$ to $\bH'$. Thus $\bH\preceq_h \bH'$.
\end{description}
\end{proof}

The second part of the previous Proposition gives the interrelation of any two homomorphism-homogeneous countable structures with the same age:

\begin{corollary}
  Let $\bA$, and $\bB$ be two weakly homomorphism-homogeneous
  structures with the same age. Then $\bA \preceq_h \bB$ and $\bB
  \preceq_h \bA$.  In particular,
any two countable homomorphism-homogeneous relational structures with
the same age are homomorphism-equivalent. \qed
\end{corollary}

\begin{remark}
  A binary relation similar to $\preceq_h$ appeared in \cite{CamNes06}
  in the context of a \Fraisse-type theorem for monomorphism
  homogeneous structures. 
\end{remark}

\section{Homomorphisms between weakly oligomorphic structures}
\label{sec:ages-weakly-olig}

In this section we collect some  results about the existence of homomorphisms
between weakly oligomorphic structures.

\begin{definition}
  Let $\cA$, $\cB$ be two classes of relational structures over a common
  signature. We say, that  $\cA$ \emph{projects} onto $\cB$ (written $\cA\to\cB$) if for every
  $\bA\in\cA$ there exists a $\bB\in\cB$, such that
  $\bA\to\bB$.
\end{definition} 
Clearly, for two relational structures $\bA$ and $\bB$, the condition $\Age(\bA)\to\Age(\bB)$ is
necessary for $\bA\to\bB$. 
In the following we will prove that if $\bA$,
is countable, and $\bB$ is weakly oligomorphic, then this
condition is also sufficient. Before coming to the prove of this claim,
let us prove a more specific result about weakly homomorphism-homogeneous relational structures. 
\begin{proposition}\label{prop:agehom}
  Let $\bA$, $\bB$ be a relational structures over the same
  signature such that $\Age(\bA)\to\Age(\bB)$, and suppose that $\bA$
  is countable, and that $\bB$ is weakly oligomorphic and weakly
  homomorphism-homogeneous. Then
  $\bA\to\bB$.
\end{proposition}
\begin{proof}
	If $\bA$ is finite, then nothing needs to be proved. So we assume that $\bA$ is countably infinite. In that case we can write the carrier $A$ as $A=\{a_0,a_1,a_2,\dots\}$. Define  $A_n:=\{a_0,\dots,a_{n-1}\}$, and let $\bA_n$ be the substructure of $\bA$ that is induced by $A_n$. 

	Let us define an equivalence relation on the homomorphisms from $\bA_n$ to $\bB$: Let $h_1,h_2: \bA_n\to\bB$. Then the two homomorphisms induce tuples 
	\[\bar{h}_1=(h_1(a_0),\dots,h_1(a_{n-1})), \text{ and } \bar{h}_2=(h_2(a_0),\dots,h_2(a_{n-1})).\]
	We call $h_1$ and $h_2$ equivalent (written $h_1\cong h_2$) if there exists a local isomorphism that maps $\bar{h}_1$ to $\bar{h}_2$. 

	We claim that there are just finitely many equivalence classes of homomorphisms from $\bA_n$ to $\bB$. Since $\bB$ is weakly oligomorphic, and weakly homomorphism-homogeneous, by  Proposition~\ref{MainThm} it follows that $\bB$ is endolocal. In other words, the $n$-ary relations that are definable by positive-existential formulae coincide with the relations that  are closed with respect to local homomorphisms. If there is a local homomorphism that maps a tuple $\bar{a}$ to a tuple $\bar{b}$, and if there is a local homomorphism that maps $\bar{b}$  to $\bar{a}$, then there is also a local isomorphism that maps $\bar{a}$ to $\bar{b}$. In this case we call $\bar{a}$ and $\bar{b}$ equivalent. Clearly, two tuples $\bar{a}$ and $\bar{b}$ are equivalent if and only if the respective  closure of $\{\bar{a}\}$ and $\{\bar{b}\}$ under local homomorphisms coincide. Since $\bB$ is weakly oligomorphic, there are just finitely many $n$-ary relations over $B$ that are closed with respect to local homomorphisms of $\bB$. Hence there are just finitely many equivalence classes of $n$-tuples over $\bB$. By this reason there can also exist only finitely many equivalence classes of homomorphisms from $\bA_n$ to $\bB$.

	Next we claim, that if $h_1,h_2:\bA_{n+1}\to\bB$ with $h_1\cong h_2$, then $h_1\restr_{A_n}\cong h_2\restr_{A_n}$. Indeed, the same local isomorphism that maps $\bar{h}_1$ to $\bar{h}_2$ will also map $\overline{h_1\restr_{A_n}}$ to $\overline{h_2\restr_{A_n}}$. 

	Next, we define a tree, whose nodes on level $n$ are all equivalence classes of homomorphisms from $\bA_n$ to $\bB$. For the equivalence class $[h]_{\cong}$ that is generated by $h:\bA_{n+1}\to \bB$, the unique lower neighbor is $[h\restr_{A_n}]_{\cong}$. 

	By the above proved claims, this tree is well-defined and finitely branching. Moreover, the tree has nodes on every level, since $\Age(\bA)\to\Age(\bB)$. Hence, by K\H{o}nig's tree-lemma, it has an infinite branch $([h_i]_{\cong})_{i\in\mathbb{N}}$. 

	It remains, to construct the homomorphism from $\bA$ to $\bB$. We proceed by induction. Our goal is to construct a tower $(\hat{h}_i)_{i\in\mathbb{N}}$ of homomorphisms $\hat{h}_i$ from  $\bA_i$ to $\bB$ such that for every $i$, there exists a local homomorphism $g_i$ such that $g_i\circ h_i=\hat{h}_i$.

	Define $\hat{h}_0:=h_0$. Suppose that $\hat{h}_i$ is already defined. Then $h_{i+1}\restr_{A_i}\cong h_i$. Let us denote $h_{i+1}\restr_{A_i}$ by $\tilde{h}_i$. Then there exists a local isomorphism $\iota$ such that $\iota\circ\tilde{h}_i=h_i$. Without loss of generality, $\dom(\iota)=\img(\tilde{h}_i)$, $\dom(g_i)=\img(h_i)$, and $\img(g_i)=\img(\hat{h}_i)$. Then $g_i\circ\iota:\img(\tilde{h}_i)\to \img(\hat{h}_i)$ is a local homomorphism, and $\img(h_{i+1})\supseteq \img(\tilde{h}_i)$. Since $\bB$ is weakly homomorphism homogeneous, $g_i\circ\iota$ extends to $\img(h_{i+1})$ (cf. the following diagram).
	\[
	\begin{psmatrix}
		\img(\tilde{h}_i) & \img(h_i) & \img(\hat{h}_i)\\
		\img(h_{i+1})      &            & \bB
		\everypsbox{\scriptstyle}
		\ncline{H->}{1,1}{2,1}<{\le}
		\ncline{H->}{1,3}{2,3}>{\le}
		\ncline{->}{1,1}{1,2}^{\iota}_{\cong}
		\ncline{->>}{1,2}{1,3}^{g_i}
		\ncline{->}{2,1}{2,3}^{g_{i+1}}
	\end{psmatrix}
	\]
	Now we define $\hat{h}_{i+1}(x):= g_{i+1}(h_{i+1}(x))$. Then for any $x\in A_i$, we calculate
	\[
	g_{i+1}(h_{i+1}(x)) = g_{i+1}(\tilde{h}_i(x))=g_i(\iota(\tilde{h}_i(x)))=g_i(h_i(x))=\hat{h}_i(x).
	\]
	Thus, $\hat{h}_{i+1}$ is indeed an extension of $\hat{h}_i$, and $g_{i+1}\circ h_{i+1}=\hat{h}_{i+1}$.

	It remains to note that the union over all $\hat{h}_i$ is a homomorphism from $\bA$ to $\bB$.
\end{proof}

\begin{proposition}\label{prop:wo-hom}
  Let $\bA$, $\bB$ be  relational structures over the same
  signature such that $\Age(\bA)\to\Age(\bB)$, and suppose that $\bA$
  is countable, and that $\bB$ is weakly oligomorphic. Then
  $\bA\to\bB$.
\end{proposition}
\begin{proof}
  First we expand the signature by all positive existential
  formulae. $\widehat{\bB}$ shall be the structure obtained from $\bB$ by
  expansion by all positive existential definitions. 

  We also expand $\bA$ to a new structure $\widehat{\bA}$ over the new
  signature. However, we interpret each additional relational symbol
  as the empty relation in $\widehat{\bA}$. 

  With this setting it is clear, that
  $\Age(\widehat{\bA})\to\Age(\widehat{\bB})$. Moreover, in $\widehat{\bB}$, every
  positive existential formula is equivalent to a quantifier-free
  positive formula. Hence, observing that $\widehat{\bB}$ is weakly
  oligomorphic, and using Proposition~\ref{MainThm}, we conclude that $\widehat{\bB}$ is weakly homomorphism- homogeneous. 

  By Proposition~\ref{prop:agehom}, it follows that there is a
  homomorphism from $h:\widehat{\bA}\to\widehat{\bB}$. Clearly, $h$ is also a
  homomorphism from $\bA$ to $\bB$.
\end{proof}

\section{Cores of homomorphism-homogeneous structures}
\label{sec:cores-homom-homog}

A finite relational structure is called a \emph{core} if each of its
endomorphisms is an automorphism. There are many ways to
generalize the definition of a core to infinite structures. Several
possibilities were explored in \cite{Bau95}. For us, the following
definition seems most reasonable: 
\begin{definition}
  A relational structure $\bC$ is called a core if every endomorphism
  of $\bC$ is an embedding.

  We say  \emph{$\bC$ is a core of $\bA$} (or $\bA$ has a core $\bC$) if $\bC$ is a core,
  $\bC\le\bA$, and there is an endomorphism $f$ of $\bA$ such that
  $\img(f)\subseteq\bC$. 
\end{definition}
For finite relational structures, the core always exist and is unique,
up to isomorphism. For infinite structures a core may exist or may not
exist. Moreover, if it exists, it may not be unique up to isomorphism.  

In this section we will employ the machinery, that was developed in
the previous section in order to study cores of
homomorphism-homogeneous structures. The crucial definition in this
section is that of a 
hom-irreducible element in some class of structures:
\begin{definition}
Let $\cC$ be a class of relational structures over the same signature
and let $\bA\in \cC$. We say that $\bA$ is \textit{hom-irreducible in
  $\cC$} if for every $\bB\in \cC$  and every homomorphism
$f:\bA\rightarrow\bB$ holds that $f$ is an embedding.
\end{definition}

For a relational structure $\bA$, by $\cC_{\bA}$ we denote the class
of all finite structures of the same type like $\bA$ that are
hom-irreducible in the age of $\bA$.

\begin{lemma}\label{LEM:AP}
	Let $\cC$ be a homo-amalgamation class, and  let $\cD$ be the class of all structures from $\cC$ that are hom-irreducible in $\cC$. If $\cC\to\cD$, then $\cD$ is a \Fraisse{} class, i.e.\ it has the HP, and the AP.
\end{lemma}
\begin{proof}
	\begin{itemize}
	\item[(HP)] Let $\bA\in \cD$, and let $\bB$ be a substructure of $\bA$ (in particular, $\bB\in\cC$). Let $\bC\in\cC$, and let $f:\bB\to\bC$ be any homomorphism. Then, since $\cC$ has the HAP and is isomorphism-closed, we have that there exist a $\bD\in \cC$, and a homomorphism $g:\bA\rightarrow \bD$ such that the following diagram commutes: 
	\[
	\begin{psmatrix}
		\bB & \bA\\
		\bC & \bD
		\everypsbox{\scriptstyle}
		\ncline{H->}{1,1}{1,2}^{\leq}
		\ncline{H->}{2,1}{2,2}^{\leq}
		\ncline{->}{1,1}{2,1}<{f}
		\ncline{->}{1,2}{2,2}<{g}
	\end{psmatrix}
	\]
	Since $\bA$ is hom-irreducible in $\cC$, it follows that $g$ is an embedding, and that $f$, being a restriction of $g$ to $\bB$, is an embedding, too. We conclude now that $\bB$ is hom-irreducible in $\cC$, and, hence, $\cD$ has the HP. 
	\item[(AP)] Let $\bA,\bB,\bC\in \cD\subseteq \cC$ and let $f_1:\bA\injto \bC$ and $f_2:\bA\injto\bB$ be embeddings. Then, by the HAP of $\cC$, there exist a $\bD\in \cC$, an embedding  $g_1:\bB\injto\bD$ and a homomorphism $g_2:\bC\to\bD$ such that the following diagram commutes:
	\[
	\begin{psmatrix}
		\bA & \bC\\
		\bB & \bD
		\everypsbox{\scriptstyle}
		\ncline{H->}{1,1}{1,2}^{f_1}
		\ncline{H->}{2,1}{2,2}^{g_1}
		\ncline{H->}{1,1}{2,1}<{f_2}
		\ncline{->}{1,2}{2,2}<{g_2}
	\end{psmatrix}
	\]
	Since $\bC\in \cD$, it follows that $g_2$ is an embedding, too. Using that $\cC\to\cD$, we obtain that there exist a structure $\widehat{\bD}\in\cD$ and a homomorphism $h:\bD\rightarrow\widehat{\bD}$. Then $\widehat{\bD}$ will be the amalgam, with $h\circ g_1$ and $h\circ g_2$ as embeddings. 
	\end{itemize}
\end{proof}

\begin{lemma}\label{hipp}
	Let $\bA$ be a weakly oligomorphic, weakly homomorphism-homogeneous relational structure. Then $\Age(\bA)\to \cC_{\bA}$.
\end{lemma}
\begin{proof}
	Consider $A^n$ and note that all tuples from $A^n$ can be quasiordered by $\bar{a}\leq_n\bar{b}$ if and only if there is a local homomorphism of $\bA$ that maps $\bar{a}$ to $\bar{b}$.

	Since $\bA$ is weakly oligomorphic, by Proposition~\ref{MainThm}, $\bA$ is endolocal. Hence the set of $n$-ary relations on $\bA$, definable by positive existential types, coincides with the filters of $(A^n,\le_n)$. It follows that the equivalence relation $\leq_n\cap \geq_n$ has finitely many equivalence classes. In particular, every properly ascending chain in $(A^n,\le_n)$ is finite and every tuple $\bar{a}$ lies below a maximal tuple $\bar{a}^{\sharp}$.

	Let $\bB$ be a finite substructure of $\bA$ with $B=\{b_1,\dots,b_n\}$.  Define $\bar{b}:=(b_1,\dots,b_n)$ and let $\bar{b}^{\sharp}$ be a maximal tuple above $\bar{b}$, say $\bar{b}^{\sharp}=(b_1^{\sharp},\dots ,b_n^{\sharp})$. Let, further, $\bD$ be the substructure of $\bA$ induced by $\{b_1^{\sharp},\dots,b_n^{\sharp}\}$. We define $f:\bB\rightarrow\bD:b_i\mapsto b_i^{\sharp}$, for $i=1,\dots, n$. Then $f$ is an epimorphism and $\bD$ is hom-irreducible in $\Age(\bA)$.  
\end{proof}

\begin{proposition}\label{retractcore}
	Let $\bA$ be a countable homomorphism-homogeneous relational structure, such that $\Age (\bA)\to\cC_{\bA}$. Then $\bA$ has a core $\bC$ with age $\cC_{\bA}$.
\end{proposition}

Before coming to the proof of Proposition~\ref{retractcore}, we need to prove a technical lemma:
\begin{lemma}\label{lemcl1} 
	Let $\bA$ be a weakly homomorphism-homogeneous relational structure, and let $\bD\leq \bA$ be hom-irreducible in $\Age(\bA)$. Further let $\widetilde{\bD}$ be a finite superstructure of $\bD$ in $\bA$. Finally, let $\widehat{\bD}\leq\bA$ be hom-irreducible in $\Age(\bA)$, and let $f:\widetilde{\bD}\twoheadrightarrow \widehat{\bD}$. Then there exists a finite substructure $\bF\leq \bA$, and an isomorphism $g:\widehat{\bD}\rightarrow \bF$ such that $\bD\leq \bF$ and the following diagram commutes: 
	\begin{equation}\label{eq:1}
\raisebox{\dimexpr-\height+\ht\strutbox}{%
\ensuremath{\begin{psmatrix}
			\widetilde{\bD} & \widehat{\bD}\\
			\bD & \bF
			\everypsbox{\scriptstyle}
			\ncline{->>}{1,1}{1,2}^{f}
			\ncline{H->}{2,1}{2,2}^{\leq}
			\ncline{<-H}{1,1}{2,1}<{\leq}
			\ncline{H->}{1,2}{2,2}<{g}>{\cong}
		\end{psmatrix}}}
	\end{equation}
\end{lemma}
\begin{proof}
	Consider the mapping $\bar{f}$ given by the following diagram:
	\[
	\begin{psmatrix}
		& f(\bD)&\\
		\bD &  &\widehat{\bD}
		\everypsbox{\scriptstyle}
		\ncline{->>}{2,1}{1,2}^{\bar{f}}
		\ncline{H->}{1,2}{2,3}^{\leq}
		\ncline{->}{2,1}{2,3}^{f\restr_{\bD}}
	\end{psmatrix}
	\] 
	and note that $\bar{f}$ is an isomorphism because $\bD$ is hom-irreducible in $\Age(\bA)$.

	Since $\bA$ is weakly homomorphism-homogeneous, it follows that $\bar{f}^{-1}$ extends to a homomorphism $g:\widehat{\bD}\to \bF$. $\bF$ can be chosen in such a way that $g$ is surjective. Note that since $\widehat{\bD}$ is hom-irreducible in $\Age(\bA)$, it follows that $g$ is an isomorphism. 
	\[
	\begin{psmatrix}
		\bD & \bF\\
		f(\bD) & \widehat{\bD}
		\everypsbox{\scriptstyle}
		\ncline{H->}{1,1}{1,2}^{\leq}
		\ncline{H->}{2,1}{2,2}^{\leq}
		\ncline{<-H}{1,1}{2,1}<{\bar{f}^{-1}}
		\ncline{H->}{2,2}{1,2}<{g}
	\end{psmatrix}
	\] 
	Let $d\in D$. Then $g(f(d))=g(\bar{f}(d))=d$, since $\bar{f}(d)\in f(D)$ and $g\restr_{f(\bD)}=\bar{f}^{-1}$. Hence, the diagram \eqref{eq:1} commutes.
\end{proof}

\begin{proof}[Proof of Proposition~\ref{retractcore}.]
	Take all finite substructures of $\bA$, and denote them by $\bE_0,\bE_1,\bE_2,\dots$.

	We will show that a core exists by constructing an endomorphism whose image has an age contained in $\cC_\bA$. This endomorphism will be obtained as the union of a tower of local homomorphisms $\varepsilon_i:\bA_i\epito\bC_i$, where $\bC_i\in\cC_{\bA}$: 
	\begin{description}
		\item[Induction basis] Define $\bA_0:=\bE_0$. Then, by assumptions, there exist a $\bC_0\in \cC_{\bA}$ and an epimorphism $\varepsilon_0:\bA_0\twoheadrightarrow \bC_0$.
		\item[Induction step] Suppose that we have constructed $\varepsilon_i:\bA_i\epito\bC_i$. Define $\bA_{i+1}:=\bA_i\cup\bE_{i+1}$. Since $\bA$ is weakly homomorphism-homogeneous, there exist a $\bD\geq \bC_i$ and an epimorphism $e:\bA_{i+1}\twoheadrightarrow\bD$ such that the following diagram commutes:
		\[
		\begin{psmatrix}
			\bA_{i+1} & \bD\\
			\bA_i & \bC_i
			\everypsbox{\scriptstyle}
			\ncline{->>}{1,1}{1,2}^{e}
			\ncline{->>}{2,1}{2,2}^{\varepsilon_i}
			\ncline{<-H}{1,1}{2,1}<{\leq}
			\ncline{H->}{2,2}{1,2}<{\leq}
		\end{psmatrix}
		\] 
		Since $\Age(\bA)\to\cC_\bA$, there exists an epimorphism $\hat{f}$ from $\bD$ to a substructure $\widehat{\bD}$ of $\bA$ that is hom-irreducible in $\Age(\bA)$. By Lemma~\ref{lemcl1}, there exists a structure $\bC_{i+1}\ge\bC_i$ that is hom-irreducible in $\Age(\bA)$, and there exists an isomorphism $g:\widehat{\bD}\to \bC_{i+1}$ such that the following diagram commutes:
		\[
		\begin{psmatrix}
			\bD & \widehat{\bD}\\
			\bC_i & \bC_{i+1} 
			\everypsbox{\scriptstyle}
			\ncline{->>}{1,1}{1,2}^{\hat{f}}
			\ncline{H->}{2,1}{1,1}<{\leq}
			\ncline{H->}{2,1}{2,2}^{\leq}
			\ncline{->}{1,2}{2,2}<{g}>{\cong}
		\end{psmatrix}
		\]
		The mapping $f:= g\circ \hat{f}$ is an epimorphism. Define $\varepsilon_{i+1}:=f\circ e$ and observe that $\varepsilon_{i+1}$ is an extension of $\varepsilon_i$. Note, further, that $\bA=\cup_{i\in\mathbb{N}} \bA_i$.

		Let $C:=\cup_{i\in\mathbb{N}} C_i$ and let $\bC\leq \bA$ be the  substructure of $\bA$ induced by $C$. Further, let $\varepsilon:=\cup_{i\in\mathbb{N}} \varepsilon_i$. Then $\varepsilon: \bA\twoheadrightarrow\bC$.
\end{description}

Instead of directly showing that $\bC$ is a core, we prove the 
stronger claim, that every homomorphism $f: \bC\rightarrow \bA$ is an embedding.

Suppose to the contrary that there exists $f:\bC\rightarrow \bA$ that
is not an embedding. Then either $f$ is not injective or $f$ is
injective, but not strong. 
\begin{description}
\item[Case 1] If $f$ is not injective, then there are $c,d\in \bC$
  such that $f(c)=f(d)$.

On the other hand, there exists an $i\in \mathbb{N}$ such that
$\{c,d\}\subseteq C_i$. Since $f\restr_{\bC_i}$ is a homomorphism and
$\bC_i$ is hom-irreducible in $\Age(\bA)$, it follows that $f\restr_{\bC_i}$ is an
embedding, and we arrive at a contradiction. 
\item[Case 2] If $f$ is injective, but not strong (i.e.\ if $f$ is a
  monomorphism, but not an embedding), then there exist a basic
  $n$-ary relation $\varrho$ of $\bA$ and a tuple
  $\bar{a}=(a_1,a_2,\dots, a_n)\in A^n\setminus \varrho$ such that
  $(f(a_1), f(a_2),\dots ,f(a_n))\in \varrho$. However, then there is
  an $i\in \mathbb{N}$ such that $\{a_1,a_2,\dots, a_n\}\subseteq
  C_i$, and $f\restr_{\bC_i}$ is an embedding, which is a
  contradiction. 
\end{description}
Summing up, we conclude that $f$ must be an embedding.

Let us finally show that $\Age(\bC)=\cC_{\bA}$:
By construction, $\Age(\bC)\subseteq\cC_{\bA}$. On the other hand, as
a core of $\bA$, every finite substructure of $\bA$ that is
hom-irreducible in $\Age(\bA)$ embeds into $\bC$. Hence
$\cC_{\bA}\subseteq \Age(\bC)$.
\end{proof}

\begin{corollary}\label{core}
  Every countable  weakly oligomorphic homomorphism-homogeneous
  structure $\bA$ has a core $\bC$ of age $\cC_{\bA}$.\qed
\end{corollary}

Before coming to the main result of this section, we need to prove a
few auxiliary results:
\begin{lemma}\label{woo}
  Let $\bA$ be a relational structure, then the following are true:
  \begin{enumerate}[(a)]
  \item\label{item:10} If $\bA$ is weakly oligomorphic, then for every $n\in\mathbb{N}$, the
    class $\Age(\bA)$ contains up to isomorphism only finitely many
    structures of cardinality $n$.
  \item\label{item:11} If for every $n\in\mathbb{N}$, the
    class $\Age(\bA)$ contains up to isomorphism only finitely many
    structures of cardinality $n$, and $\bA$ is homomorphism-homogeneous, then $\bA$ is weakly oligomorphic.
  \end{enumerate}
\end{lemma}
\begin{proof}
\begin{description}
\item[About (\ref{item:10})] Suppose, there are infinitely many isomorphism
  classes of substructures of cardinality $k$. Let
  $(\bB_i)_{i\in\mathbb{N}}$ be a sequence of distinct representatives of
  isomorphism classes of substructures of $\bA$ of cardinality
  $k$. Let us enumerate the 
  elements of $B_i$ like $B_i=\{b_{i,1},\dots,b_{i,k}\}$. Consider the
  tuples $\bar{b}_i=(b_{i,1},\dots,b_{i,k})$, and the relations
  $\varrho_i:=\{\bar{c}\mid \pTp_\bA(\bar{b}_i)\subseteq\pTp_\bA(\bar{c})\}$.

  Then for any two
  distinct $i,j$ from $\mathbb{N}$ we have that 
  $\pTp_\bA(\bar{b}_i)\neq\pTp_\bA(\bar{b}_j)$, and hence
  $\varrho_i\neq\varrho_j$. This way we have infinitely many distinct
  $k$-ary relations on $\bA$ that can be defined by sets of positive
  existential formulae over $\bA$ --- contradiction.

  \item[About (\ref{item:11})] Equip the $n$-tuples over $A$ with the following
  quasiorder: $\bar{a}\le \bar{b}$ if there is a local homomorphism
  that maps $\bar{a}$ to $\bar{b}$. Since $\bA$ is homomorphism-homogeneous, this is the case if and only if $\bar{b}$ is in the
  invariant relation of $\End(\bA)$ generated by $\bar{a}$. 

  We will show that $\End(\bA)$ is oligomorphic. Suppose, it is not. Then there is a $k\ge 1$ and a sequence $(\bar{b}_i)_{i\in\mathbb{N}}$ of $k$-tuples over $A$, such that for all distinct natural numbers  $i$ and $j$ we have that $\bar{b}_i$ and $\bar{b}_j$ generate different invariant relations of $\End(\bA)$. If this is so, then by the infinite pigeon hole principle, there exists an $n\le k$ and a sequence $(\bar{c}_i)_{i\in\mathbb{N}}$ of irreflexive $n$-tuples over $A$ such that any two tuples generate different invariant relations of $\End(\bA)$. Let $M$ be the number of isomorphism classes of substructures of cardinality $n$ in $\bA$. Then the number of invariant $n$-ary relations generated by $n$-ary irreflexive tuples is bounded from above by $M\cdot n!$ --- contradiction. Hence $\End(\bA)$ is oligomorphic. 

  Since all relations definable by sets of positive existential
  formulae over $\bA$ are invariant under $\End(\bA)$, it follows
  that $\bA$ is weakly oligomorphic.
\end{description}
\end{proof}
An immediate consequence  of the previous Lemma is that if $\bA$ is a homomorphism-homogeneous relational structure over a finite signature, then $\bA$ is weakly oligomorphic. This, together with the characterization of the ages of countable homomorphism-homogeneous structures, gives a rich source of  weakly oligomorphic structures, since any reduct of a countable weakly oligomorphic structure will again be weakly oligomorphic.

\begin{lemma}\label{woo2}
  Let $\bC$ be a homogeneous core. Then $\bC$ is weakly oligomorphic
  if and only if it is oligomorphic.
\end{lemma}
\begin{proof}
	Obviously, if a structure is oligomorphic, then it is weakly oligomorphic, too.

	Suppose now that $\bC$ is weakly oligomorphic. Take $\bar{a}=(a_1,\dots, a_n)\in C^n$ and $\bar{b}=(b_1\dots,b_n)$ such that there exists an $f\in \End(\bC)$, such that $f(\bar{a})=\bar{b}$. Since $\bC$ is a core, we conclude that $f$ is an embedding, and, therefore, $a_i\mapsto b_i$, for $i=1,\dots ,n$ is a local isomorphism. Since $\bC$ is homogeneous, it follows that there is a $g\in \Aut(\bC)$ such that $g(\bar{a})=\bar{b}$, so $\bar{b}$ is in the $n$-orbit of $\Aut(\bC)$ that is generated by $\bar{a}$. This implies that $\bC$ is oligomorphic.
\end{proof}

The following result links homomorphism-homogeneous structures to
homogeneous structures. 
\begin{theorem}\label{hh-core}
	Let $\bA$ be a weakly oligomorphic countable homomorphism-homogeneous structure. Then $\bA$ has a core $\bF$ that is isomorphic to the \Fraisse-limit of $\cC_{\bA}$. Moreover, $\bF$ is oligomorphic. 
\end{theorem}
\begin{proof}
  Let $\bF$ be the \Fraisse-limit of
  $\cC_\bA$. In particular, $\bF$ is homogeneous. 

  Since $\bA$ is weakly oligomorphic and
  $\cC_\bA\to\Age(\bA)$, we conclude from
  Proposition~\ref{prop:wo-hom}, that $\bF\to\bA$. Since
  $\Age(\bF)=\cC_{\bA}$, every homomorphism from $\bF$ to $\bA$ is an
  embedding. So we can assume without loss of generality that
  $\bF\le\bA$.   

  Every local homomorphism of $\bF$ is an embedding. Hence $\bF$ is a
  homomorphism-homogeneous core. Since $\cC_{\bA}\subseteq\Age(\bA)$, and
  since $\bA$ is weakly oligomorphic, from Lemma~\ref{woo}, it follows
  that $\bF$ is weakly oligomorphic. Since $\bF$ is a homogeneous
  core, from Lemma~\ref{woo2}, it follows that $\bF$ is oligomorphic.

  It remains to show that $\bA\to\bF$. 
  By Corollary \ref{core}, $\bA$ has a core $\bC$ such that
  $\Age(\bC)=\cC_{\bA}$. Hence, by Proposition~\ref{prop:wo-hom}, it
  follows that $\bC\to\bF$. 
\end{proof}

\begin{corollary}
	Every countable, weakly oligomorphic, homomorphism-homogeneous structure $\bA$ has, up to isomorphism, a unique homomorphism-homogeneous core $\bF$. Moreover, $\bF$ is oligomorphic and homogeneous.
\end{corollary}
\begin{proof}
  Theorem \ref{hh-core} guaranties the existence of 
  $\bF$. Indeed, $\bF$ is homomorphism-homogeneous, because every
  local homomorphism of $\bF$ is an embedding. 

  Suppose that $\bF'$ is another such core. Then
  $\cC_{\bA}\subseteq \Age(\bF')$.
  On the other hand, $\Age(\bA)\to\cC_{\bA}$, hence
  $\Age(\bF')\to\cC_{\bA}$. Hence any substructure of $\bF$ that is
  not hom-irreducible in $\Age(\bA)$, homomorphically maps to a
  hom-irreducible element. This defines a local homomorphism that is
  not an embedding. By the homomorphism-homogeneity of $\bF'$, this
  extends to an endomorphism, that is not an embedding ---
  contradiction. Thus $\Age(\bF')=\cC_{\bA}$, and every local
  homomorphism is an embedding. From this it follows that $\bF'$ is
  weakly homogeneous, and hence homogeneous. From \Fraisse's
  theorem we conclude that $\bF\cong\bF'$. 
\end{proof}

\begin{example}\label{ex1}
	It is well-known that the Rado-graph is homogeneous, homomorphism-homogeneous, oligomorphic, and weakly oligomorphic. Its age consists of all finite graphs. A finite graph is hom-irreducible in the class of all finite graphs if and only if it is a complete graph. Hence, the homomorphism-homogeneous core of the Rado-graph is the countably infinite complete graph $K_\omega$. Note that in fact  $K_\omega$ is a retract of the Rado graph (cf. \cite{BonDel00}).
\end{example}
The previous example can be generalized to the class of all homomorphism-homogeneous graphs:
\begin{example}\label{ex2}
	Let $\bG$ be a countably infinite homomorphism-homogeneous graph. Then all hom-irreducible graphs in $\Age(\bG)$ are complete graphs. This can be seen by the following argument: If $\bG$ is itself a complete graph, then nothing need to be proved. So suppose that $\bG$ is not complete. Let $\bA\in\Age(\bG)$ be a finite non-complete subgraph --- say, $a,b\in A$ induce a non-edge in $\bA$. Let $\overline{K_2}$ be the graph with two vertices and no edge, and let $K_1$ be the trivial graph with just one vertex and no edge. Suppose, the vertex set of $\overline{K}_2$ is $\{1,2\}$. Then $\iota:\overline{K}_2\injto\bA$ defined by $1\mapsto a$, $2\mapsto b$ is a graph-embedding. Let $h$ be the unique homomorphism from $\overline{K}_2\to K_1$. Then, since $\Age(\bG)$ has the HAP, there exists a finite graph $\bB\in\Age(\bG)$, a homomorphism $\hat{h}:\bA\to\bB$, and an embedding $\hat\iota:K_1\injto\bB$, such that the following diagram commutes:
	\[
	\begin{psmatrix}
		[name=K1] K_1 & [name=B]\bB \\
		[name=K2]\overline{K}_2 & [name=A] \bA
		\everypsbox{\scriptstyle}
		\ncline{H->}{K2}{A}^\iota
		\ncline{H->}{K1}{B}^{\hat\iota}
		\ncline{->}{K2}{K1}<h
		\ncline{->}{A}{B}>{\hat{h}}
	\end{psmatrix}
	\]
	In particular, $\hat{h}(a)=\hat{h}(b)$. Consequently, $\bA$ is not hom-irreducible in $\Age(\bG)$. We conclude that every countable homomorphism-homogeneous graph has a core that is a complete graph.
	
	Note that the countable homomorphism-homogeneous graphs are still not completely classified. 
\end{example}

\section{Weak oligomorphy and $\omega$-categoricity}\label{sec:omega-categ-substr}

In this section we will create a link from weakly oligomorphic structures to $\omega$-categorical structures. Let us start by recalling some classical notions and results from model theory, and by proving some additional auxiliary results:
\begin{definition}
	A first order theory is called $\omega$-categorical if it has up to isomorphism exactly one countably infinite model. A countably infinite structure $\bA$ is called $\omega$-categorical if $\Th(\bA)$ is $\omega$-categorical.
\end{definition}
The notions of oligomorphy and $\omega$-categoricity are linked by the Engeler-Ryll-Nardzewski-Svenonius Theorem  (cf. \cite[(2.10)]{Cam90}):
\begin{theorem}[Engeler, Ryll-Nardzewski, Svenonius]\label{thm:ryll-nardzewsky}
  Let $\bA$ be a countably infinite structure. Then $\bA$ is
  $\omega$-categorical if and only if it is oligomorphic.
\end{theorem}
In the previous section we linked weakly oligomorphic homomorphism-homogeneous structures with oligomorphic homogeneous structures. The following Theorem makes a similar link between weakly oligomorphic structures and $\omega$-categorical structures. 
\begin{theorem}\label{coreofwo}
	Let $\bA$ be a countable weakly oligomorphic relational  structure. Then $\bA$, is homomorphism-equivalent to a finite or $\omega$-categorical structure $\bF$. Moreover, $\bF$ embeds into  $\bA$. 
\end{theorem}
\begin{proof}
	Let $\widehat{\bA}$ be the structure that is obtained by expanding $\bA$ by all positive existentially definable relations over $\bA$. Clearly,  $\widehat{\bA}$ is weakly oligomorphic, too.

	In $\widehat{\bA}$ every positive existential formula is equivalent to a positive quantifier-free formula. Hence, by Proposition~\ref{MainThm}, $\widehat{\bA}$ is homomorphism-homogeneous. With  Theorem~\ref{hh-core} we conclude that $\widehat{\bA}$ has a homomorphism-homogeneous core  $\widehat{\bF}$ that is oligomorphic, and homogeneous.

	Let $\bF$ be the reduct of $\widehat{\bF}$ to the signature of  $\bA$. Then $\bF$ is oligomorphic, and since $\widehat{\bA}$ and $\bA$ have the same endomorphisms, $\bF$ is still homomorphism-equivalent to $\bA$. 

	If $\bF$ is countably infinite, then, by Theorem~\ref{thm:ryll-nardzewsky}, it is $\omega$-categorical. 
\end{proof}

\section{Positive existential theories of weakly oligomorphic structures}\label{sec:posit-exist-theor}

The Engeler-Ryll-Nardzewski-Svenonius Theorem (cf. Theorem \ref{thm:ryll-nardzewsky})
can be understood as a characterization of the first order theories of
countable oligomorphic structures. Using Theorem~\ref{coreofwo}, we
can give a similar characterization of the positive existential
theories of weakly oligomorphic structures.
\begin{theorem}
Let $T$ be a set of positive existential propositions. Then the following are equivalent:
\begin{enumerate}
\item $T$ is the positive existential theory of a countable weakly oligomorphic structure.
\item $T$ is the positive existential part of an $\omega$-categorical theory.
\item $T$ is the positive existential theory of a countable oligomorphic structure.
\end{enumerate} 
\end{theorem}
\begin{proof}
From Theorem~\ref{thm:ryll-nardzewsky}, it follows that statements 2 and 3
are equivalent. Obviously, from 3 follows 1, so, to complete the
proof, it is left to show that from 1 follows 3:

Let $T$ be the positive existential theory of a countable weakly
oligomorphic structure $\bA$. Then, by Theorem~\ref{coreofwo},
$\bA$ is homomorphism-equivalent to a finite or  $\omega$-categorical structure
$\bF$. If $\bF$ is finite, then it is homomorphism-equivalent to an
$\omega$-categorical structure $\widehat{\bF}$ (take $\omega$ disjoint
copies of $\bF$; this structure surely is oligomorphic and hence
$\omega$-categorical; moreover, $\bF$ is a retract of $\widehat{\bF}$). 

Clearly, two homomorphism-equivalent structures have the same positive
existential theories.
\end{proof}

By the Engeler-Ryll-Nardzewski-Svenonius Theorem two oligomorphic structures of the same type are isomorphic if and only if they have the same first order theory. The following is an analogous statement for weakly oligomorphic structures. 
\begin{theorem}\label{woRN}
	Let $\bA$ and $\bB$ be two countable weakly oligomorphic structures. Then the following are equivalent:
	\begin{enumerate}
	\item $\bA$ and $\bB$ are homomorphism-equivalent,
	\item $\bA$ and $\bB$ have the same positive existential theory,
	\item $\Age(\bA)\leftrightarrow\Age(\bB)$,
	\item $\CSP(\bA)=\CSP(\bB)$.
	\end{enumerate}
	Here
	\[ \CSP(\bA):=\{\bC\mid \bC\text{ finite, } \bC\to\bA\}.\]
\end{theorem}
Before coming to the proof of Theorem~\ref{woRN}, we need to prove some additional auxiliary results:
\begin{lemma}\label{lem_wo_real} 
	Let $\bA$ be a weakly oligomorphic structure over the signature $R$, and let $\Psi$ be a positive existential type. If every finite subset of $\Psi$ is realized in $\bA$, then $\Psi$ is realized in $\bA$.
\end{lemma}
\begin{proof} 
	Suppose that every finite subset of $\Psi$ is realized in $\bA$, but $\Psi$ is not. Following we will define a  sequence  $(\varphi_i)_{i\in\mathbb{N}}$ of formulae from $\Psi$, and a sequence $(\bar{d}_i)_{i\in\mathbb{N}}$ such that $\bar{d}_i\in\varphi_j^\bA$ for $1\le j \le i$, but $\bar{d}_i\notin \varphi_{i+1}^\bA$.

	Let $\varphi_0\in\Psi$. Then there exists a $\bar{d}_0\in A^m$ such that $\bar{d}_0\in \varphi_0^\bA$. Suppose that $\varphi_i$, and $\bar{d}_i$ are defined already. By assumption, $\bar{d}_i$ does not  realize $\Psi$. Let $\varphi_{i+1}\in\Psi$ such that  $\bar{d}_i\notin\varphi_{i+1}^\bA$. Again, by assumption, the set $\{\varphi_1,\dots,\varphi_{i+1}\}$ is realized in $\bA$. Define $\bar{d}_{i+1}\in A^m$ to be a tuple that realizes $\{\varphi_1,\dots,\varphi_{i+1}\}$.

	By construction, the sets $\Psi_k=\{\varphi_1,\dots,\varphi_k\}$ define distinct non-empty relations in $\bA$. However, this is in contradiction with the assumption that $\bA$ is weakly oligomorphic.

	We conclude that $\Psi$ is realizable.
\end{proof}

\begin{lemma}\label{lem:wo_age} 
	Let $\bA$ and $\bB$ be relational structures over the same  signature. Suppose that $\bB$ is weakly oligomorphic and that $\bA$ and $\bB$ have the same positive existential theories. Then $\Age(\bA)\to\Age(\bB)$.
\end{lemma}
\begin{proof}
	Let $\bC$ be a finite substructure of $\bA$, and let  $C=\{c_1,\dots,c_n\}$ be its carrier. Define  $\bar{c}:=(c_1,\dots,c_n)$. Then every finite subset of  $\pTp_\bA(\bar{c})$ is realized in $\bB$. Since $\bB$ is weakly oligomorphic, by Lemma~\ref{lem_wo_real}, we get that there is a tuple $\bar{d}\in B^n$ that realizes $\pTp_\bA(\bar{c})$. Let $D=\{d_1,\dots,d_n\}$, and let $\bD$ be the substructure of $\bB$ induced by $D$. Then the mapping $f:\bC\to\bD$ given by $c_i\mapsto d_i$ is a homomorphism. This shows that $\Age(\bA)\to\Age(\bB)$
\end{proof}

\begin{proof}[Proof of Theorem~\ref{woRN}]
	($1\Rightarrow 4$) Clear.
	
	($4\Rightarrow 2$) It is well known (and easy to see) that for every positive primitive proposition $\varphi$ there exists a finite relational structure $\bA_\varphi$ such that for any relational structure $\bC$ of the given type, we have $\bC\models\varphi$ if and only if $\bA_\varphi\to\bC$. Thus from $\CSP(\bA)=\CSP(\bB)$ it follows that $\bA$ and $\bB$ have the same positive primitive theory. However, this is the case if and only if $\bA$ and $\bB$ have the same positive existential theory.
	
	($2\Rightarrow 3$) This is a direct consequence of Lemma~\ref{lem:wo_age}.
	
	($3\Rightarrow 1$) This follows from Proposition~\ref{prop:wo-hom}.
\end{proof}
Let us recall a result from \cite{MasPec11}:
\begin{proposition}[{\cite[Th.3.5]{MasPec11}}]
	Let $\bA$ be a countable weakly oligomorphic structure, and let $\bB$ be a countable relational structure such that $\bB\models\Th(\bA)$. Then $\bB$ is weakly oligomorphic, too.\qed
\end{proposition} 
We can combine this with Theorem~\ref{woRN} to obtain:
\begin{corollary}
	Let $T$ be the complete first order theory of a weakly oligomorphic structure. Then all countable models of $T$ are homomorphism-equivalent.\qed
\end{corollary}

\section*{Concluding remarks}
A first order theory, for which all countable models are homomorphism-equivalent can rightfully be called \emph{weakly $\omega$-categorical}. So we can say that the first order theory of a weakly oligomorphic structure always is weakly $\omega$-categorical. The reverse implication is likely not to hold. However, so far we are not aware of a counter-example. Therefore we formulate the following problem:
\begin{problem}
	What are the countable (relational) structures whose first order theory is weakly $\omega$-categorical?
\end{problem}

In \cite{MarBodHil09}, Martin, Bodirsky, and Hils, among other things, characterized all countable structures whose positive existential theory coincides with the positive existential theory of an $\omega$-categorical structure. By Theorem~\ref{coreofwo} weakly oligomorphic structures fall under  their classification. In view of Theorem~\ref{woRN}, it should be possible to use the result from \cite{MarBodHil09} in the proof of Theorem~\ref{coreofwo}. However, we decided against this way, since our elementary techniques give a better idea of the $\omega$-categorical structure inside of a weakly oligomorphic structure. Indeed, in many cases (in particular for homomorphism homogeneous structures) our technique does not only yield the existence of an $\omega$-categorical substructure but a concrete description (cf. Example~\ref{ex2}).

\end{document}